\documentclass[12pt,fceqn]{article}
\usepackage{amsmath}
\usepackage{amsfonts}
\usepackage{cite}
\usepackage{amssymb,graphics,graphicx,subfigure,caption2,rotating}
\usepackage{amsmath, latexsym}
\usepackage[amsmath,thmmarks]{ntheorem}
\usepackage{color}
\numberwithin{equation}{section}
\newtheorem{theorem}{Theorem}[section]
\newtheorem{lemma}{Lemma}[section]
\newtheorem{remark}{Remark}[section]

\newtheorem{cor}{Corollary}[section]
\newtheorem{prop}{Proposition}[section]

\theoremsymbol{$\square$}
\newtheorem*{proof}{Proof.}  
\newcommand{\beq}{\begin{equation}}
\newcommand{\eeq}{\end{equation}}
\newcommand{\beqn}{\begin{eqnarray}}
\newcommand{\eeqn}{\end{eqnarray}}
\date{}%
\begin{document}

\date{}
\title{The error and perturbation bounds of the general absolute value equations\thanks{This
research was supported by National Natural Science Foundation of
China (No.11961082) and Cross-integration Innovation team of modern
Applied Mathematics and Life Sciences in Yunnan Province,
China(No.202405AS350003).}}
\author{Shi-Liang Wu\thanks{Corresponding author:
slwuynnu@126.com}, Cui-Xia Li\thanks{lixiatkynu@126.com}\\
{\small{\it $^{\dagger,\dagger}$School of Mathematics, Yunnan Normal University,}}\\
{\small{\it Kunming, Yunnan, 650500, P.R. China}}\\
{\small{\it $^{\dagger}$Yunnan Key Laboratory of Modern Analytical Mathematics and Applications,}}\\
{\small{\it  Yunnan Normal University, Kunming, Yunnan, 650500, P.R.
China}} }
 \maketitle
\begin{abstract}
To our knowledge,  the error and perturbation bounds of the general
absolute value equations are not discussed. In order to fill in this
study gap, in this paper, by introducing a class of absolute value
functions, we study the error and perturbation bounds of two types
of the general absolute value equations (AVEs): $Ax-B|x|=b$ and
$Ax-|Bx|=b$. Some useful error bounds and perturbation bounds of the
above two types of absolute value equations are provided. Without
limiting the matrix type, some computable estimates for the above
upper bounds are given. By applying the absolute value equations, a
new approach for some existing perturbation bounds  of the linear
complementarity problem (LCP) in (SIAM J. Optim., 18 (2007), pp.
1250-1265) is provided. Some numerical examples for the AVEs from
the LCP are given to show the feasibility of the perturbation
bounds.


\end{abstract}

\section{Introduction} Consider the following two types of the general absolute value equations
(AVEs)
\begin{equation}\label{p1}
Ax-B|x|=b
\end{equation}
and
\begin{equation}\label{eq:2}
Ax-|Bx|=b,
\end{equation}
where $A,B\in \mathbb{R}^{n\times n}$ and $b\in \mathbb{R}^{n}$,
$|\cdot|$ denotes the componentwise absolute value of the vector.
The AVEs (\ref{p1}) and (\ref{eq:2}), respectively, were introduced
in \cite{Rohn04} by Rohn and \cite{Wu21a} by Wu. Clearly, when $B=I$
in (\ref{p1}) and  (\ref{eq:2}), where $I$ denotes the identity
matrix, the AVEs (\ref{p1}) and (\ref{eq:2}) reduce to the standard
absolute value equations
\begin{equation}\label{p3}
Ax-|x|=b,
\end{equation}
which was considered in \cite{Mangasarian06} by Mangasarian and
Meyer.

Over these years, the AVEs (\ref{p1}), (\ref{eq:2}) and  (\ref{p3})
have excited much interest since they often occur in many
significant mathematical programming problems, including linear
programs, quadratic programs, bimatrix game, linear complementarity
problem (LCP),  see \cite{Mangasarian06,Mangasarian07, Rohn89,
Prokopyev09} and references therein. For instance, the AVEs
(\ref{p1}) is equivalent to the LCP of determining a vector $z\in
\mathbb{R}^{n}$ such that
\begin{equation}\label{eq:14}
w=Mz+q\geq0,\ z\geq0 \ \mbox{and}\ z^{T}w=0 \ \mbox{with}\ M\in
\mathbb{R}^{n\times n}\ \mbox{and}\  q\in \mathbb{R}^{n}.
\end{equation}
By using $z=|x|+x$ and $w=|x|-x$ for (\ref{eq:14}), the AVEs
(\ref{p1}) is obtained with $A=I+M$,  $B=I-M$ and $b=-q$.

Although the AVEs in \cite{Mangasarian06} is a NP-hard problem, so
far, a large number of theoretical results, numerical methods and
applications have been extensively excavated. For instance, among
the theoretical results, in spite of determining the existence of
the solution of the AVEs in \cite{Mangasarian07,Chen21} is NP-hard,
and checking whether the AVEs has unique or multiple solutions in
\cite{Prokopyev09} is also
NP-complete, 
there still exist some very important conclusions, in particular,
some sufficient and necessary conditions for ensuring the existence
and uniqueness of the solutions of  the AVEs (\ref{p1}),
(\ref{eq:2}) and (\ref{p3})  were established for any $b\in
\mathbb{R}^{n}$, see \cite{Wu21b, Wu18, Wu21a, Mezzadri20}.

Likewise, solving the AVEs  in \cite{Mangasarian07} is NP-hard as
well. It may be due to the fact that the AVEs contains a non-linear
and non-differential absolute value operator. Even so, some
efficient numerical methods have been developed, such as the
generalized Newton method \cite{Mangasarian09}, Newton-based matrix
splitting method \cite{Zhou21}, the exact and inexact
Douglas-Rachford splitting method \cite{Chen22}, the Picard-HSS
method \cite{Salkuyeh14}, the sign accord method \cite{Rohn09}, the
concave minimization method \cite{Mangasarian07b}, the
Levenberg-Marquardt method \cite{Iqbal15}, and so on.

As an important application of the AVEs, for all we know, the AVEs
was first viewed as a very effective tool to gain the numerical
solution of the LCP in \cite{Bokhoven80}, called the modulus method.
At present, this numerical method has achieved rapid development and
its many various versions were proposed, see \cite{Bai10, Wu22} and
references therein. Since the modulus method has the superiorities
of simple construction and quick convergence behavior, it is often
regarded as a top-priority method for solving the large-scale and
sparse complementarity problem (CP).

In addition to the above aspects about the AVEs, another very
important problem is the sensitivity and stability analysis of the
AVEs, i.e., how the solution variation is when the data is
perturbed. More specifically, when $\Delta A, \Delta B
~\mbox{and}~\Delta b$ are the perturbation terms of $A,
B~\mbox{and}~b$ in (\ref{p1}) and (\ref{eq:2}), respectively, how do
we characterize the change in the solution of the perturbed AVEs.
With respect to this regard, to our knowledge, the perturbation
analysis of the AVEs (\ref{p1}) and (\ref{eq:2}) has not been
discussed. In addition, for the error bound of AVEs, under the
assumption of strongly monotone property, a global projection-type
error bound was provided in \cite{Chen21}.  Obviously, this kind of
conditional projection-type error bound frequently has great
limitation.\cite{Zamani} only provides some upper bounds for the
error bound of the AVE (\ref{p1}) with $B=I$. Therefore, based on
these considerations, in this paper, we in-depth discuss the error
bounds and the perturbation bounds of the AVEs. The contributions
are given below:
\begin{itemize}
\item  Firstly, by introducing a class of the absolute value
functions, the framework of error bounds for the AVEs are presented.
Without limiting the matrix type, some computable estimates for
their upper bounds are given. These bounds are sharper than the
existing bounds in \cite{Chen21} under certain conditions.
\item Secondly, we establish the framework of perturbation
bounds for the AVEs and present  some computable upper bounds. It is
pointed out that when the nonlinear term $B|x|$ in (\ref{p1}) is
vanished, the presented perturbation bounds reduce to the classical
perturbation bounds for the linear systems $Ax=b$, like Theorem 1.3
in numerical linear algebra textbooks \cite{Sun01} and Theorem 2.1
\cite{Skeel79}.

\item  Thirdly,
a new approach for some existing perturbation bounds of the LCP in
\cite{Chen07}  is provided as well.

\item Of course, finally,
to show the efficiency of some proposed bounds, some numerical
examples for the AVEs from the LCP are investigated.
\end{itemize}

The rest of the article is organized as follows. Section 2 provides
the framework of error bounds for the AVEs by introducing a class of
absolute value functions.  In Section 3, some perturbation bounds
for the AVEs are provided. In Section 4, as an application, a new
approach for some existing perturbation bounds of the LCP in
\cite{Chen07}  is provided. In Section 5, some numerical examples
for the AVEs from the LCP are given to show the feasibility of some
perturbation bounds. Finally, in Section 6, we end up with this
paper with some conclusions.

\vspace{0.3cm} Finally, to end this section, we remind some
notations, definitions  and conclusions in
\cite{Berman,Wu21a,Wu21b,Horn}, which will be used in the later
discussion.

Let $A = (a_{ij})$ and $N=\{1,2,\ldots,n\}$. Then we denote
$|A|=(|a_{ij}|)$. $A =(a_{ij})$ is called an $M$-matrix if
$A^{-1}\geq0$ and $a_{ij}\leq0$ ($i\neq j$) for $i,j \in N$; an
$H$-matrix if its comparison matrix $\langle A\rangle$ (i.e.,
$\langle a_{ii}\rangle=|a_{ii}|, \langle a_{ij} \rangle=-|a_{ij}|$
$i\neq j$ for $i,j \in N$) is an $M$-matrix; an $H_{+}$-matrix if
$A$ is an $H$-matrix with $a_{ii} > 0$ for  $i\in N$; a $P$-matrix
if all principal minors of $A$ are positive. Let $\rho(\cdot)$,
$\sigma_{\min}$ and $\sigma_{\max}$ denote the spectral radius, the
smallest singular value and the largest singular value of  matrix,
respectively. For two vectors $q, e\in \mathbb{R}^{n}$, by $q_{+}$
and $q_{-}$ we denote $q_{+}=\max\{0,q\}$, $q_{-}=\max\{0,-q\}$ and
$e=(1,1,\ldots,1)^{T}$. The norm $\|\cdot\|$ means $p$-norm, i.e.,
$\|\cdot\|_{p}$ with $p\geq1$.


The AVEs (\ref{p1}) has a unique solution for any $b\in
\mathbb{R}^{n}$ if and only if $A-BD$ is nonsingular for any
diagonal matrix $D = diag(d_{i})$ with $d_{i}\in [-1, 1]$, see
Theorem 3.2 in \cite{Wu21b}; the AVEs (\ref{eq:2}) has a unique
solution for any $b\in \mathbb{R}^{n}$ if and only if $A-DB$ is
nonsingular for any diagonal matrix $D = diag(d_{i})$ with $d_{i}\in
[-1, 1]$, see Theorem 3.3 in \cite{Wu21a}.

\section{Error bound} In this section, without further illustration, we always assume that
the AVEs (\ref{p1}) and (\ref{eq:2}) have the unique solution. Under
this premise,  we can give the framework of error bounds on the
distance between the approximate solutions and the exact solutions
of the AVEs (\ref{p1}) and AVEs (\ref{eq:2}), respectively.

\subsection{Framework of error bounds for AVEs}
In this subsection, the framework of error bounds for the AVEs is
obtained. To achieve our goal, the following absolute value function
is introduced, see Lemma 2.1.
\begin{lemma}
Let  $a=(a_{1},a_{2},\ldots,a_{n})^{T},
b=(b_{1},b_{2},\ldots,b_{n})^{T}$ be any two vectors in
$\mathbb{R}^{n}$. Then there exist $d_{i}\in [-1,1]$ such that
\begin{equation}\label{eq:21}
|a_{i}|-|b_{i}|=d_{i}(a_{i}-b_{i}), \mbox{for\ all}\ i=1,2,\ldots,n.
\end{equation}
\end{lemma}
\begin{proof} Its proof is straightforward, which is omitted.
\end{proof}

Let
\[
r(x)=Ax-B|x|-b.
\]
Clearly,  $x^{\ast}$ is a solution of the AVEs (\ref{p1}) if and
only if $r(x^{\ast}) =0$. The function $r(x)$ is called the natural
residual of the AVEs (\ref{p1}). Let $x^{\ast}$ be the unique
solution of the AVEs (\ref{p1}). Then from Lemma 2.1 we have
\begin{align*}\label{31}
r(x)=& Ax-B|x|-b-(Ax^{\ast}-B|x^{\ast}|-b)\\
=& A(x-x^{\ast})-B(|x|-|x^{\ast}|)\\
=& A(x-x^{\ast})-B\tilde{D}(x-x^{\ast})\\
=& (A-B\tilde{D})(x-x^{\ast}),
\end{align*}
where $\tilde{D}=diag(\tilde{d}_{i})$ with $\tilde{d}_{i}\in
[-1,1]$, which promptly results in the error bounds for the AVEs
(\ref{p1}), see Theorem 2.1.

\begin{theorem} Let $x^*$ be the unique solution of AVEs (\ref{p1}). Then for any $x\in
\mathbb{R}^{n}$,
\begin{equation}\label{eq:22}
\frac{1}{\underline{\alpha}} \|r(x)\|\leq\|x-x^{\ast}\|\leq
\overline{\alpha}\|r(x)\|,
\end{equation}
where
\begin{equation*}
\underline{\alpha}=\max\{\|A-BD\|: D = \mbox{diag}(d_i)~
\mbox{with}~ d_i \in[-1,1]\}\end{equation*} \mbox{and}
\begin{equation*} \overline{\alpha}=\max\{\|(A-BD)^{-1}\|: D =
\mbox{diag}(d_i)~ \mbox{with}~ d_i \in[-1,1]\}.
\end{equation*}
\end{theorem}

In Theorem 2.1, $\underline{\alpha}\overline{\alpha}\geq1$. In fact,
\begin{align*}
\underline{\alpha}\overline{\alpha}&\geq\max\{\|A-BD\|\|(A-BD)^{-1}\|:D
= diag(d_i)~ \mbox{with}~ d_i
\in[-1,1]\}\\
&\geq\max\{\|(A-BD)(A-BD)^{-1}\|:D = diag(d_i)~ \mbox{with}~ d_i
\in[-1,1]\}\\
&=1.
\end{align*}
Similarly, the following involved results, i.e., Theorem 2.2 and
Corollary 2.1, are completely analogical.

By using the same technique for the AVEs (\ref{eq:2}), we have

\begin{theorem} Let $x^*$ be the unique solution of AVEs (\ref{eq:2}). Then for any $x\in
\mathbb{R}^{n}$,
\begin{equation}\label{eq:23}
\frac{1}{\underline{\beta}} \|r(x)\|\leq\|x-x^{\ast}\|\leq
\overline{\beta}\|r(x)\|,
\end{equation}
where
\begin{equation*}
\underline{\beta}=\max\{\|A-DB\|: D =\mbox{diag}(d_i)~ \mbox{with}~
d_i \in[-1,1]\}\end{equation*} \mbox{and} \begin{equation*}
\overline{\beta}=\max\{\|(A-DB)^{-1}\|: D = diag(d_i)~ \mbox{with}~
d_i \in[-1,1]\}.
\end{equation*}
\end{theorem}

When $B=I$ in Theorem 2.1 or Theorem 2.2,  the error bounds for the
AVEs (\ref{p3}) can be obtained, see Corollary 2.1.

\begin{cor} Let $x^*$ be the unique solution of AVEs (\ref{p3}). Then for any $x\in
\mathbb{R}^{n}$,
\begin{equation}\label{eq:24}
\frac{1}{\underline{\gamma}} \|r(x)\|\leq\|x-x^{\ast}\|\leq
\overline{\gamma}\|r(x)\|,
\end{equation}
where
\begin{equation*}
\underline{\gamma}=\max\{\|A-D\|: D = diag(d_i)~ \mbox{with}~ d_i
\in[-1,1]\}\end{equation*} \mbox{and} \begin{equation*}
\overline{\gamma}=\max\{\|(A-D)^{-1}\|: D = diag(d_i)~ \mbox{with}~
d_i \in[-1,1]\}.
\end{equation*}
\end{cor}

\begin{remark} Clearly,  Corollary 2.1 provides the lower and upper error bounds for
the AVEs (\ref{p3}). Meanwhile, it is not difficult to find that the
right side in (\ref{eq:24}) is equal to Theorem 7 in \cite{Zamani},
which is a main result in \cite{Zamani}. Not only that, we present a
lower error bound for the AVEs (\ref{p3}). In this way, we remedy
their work.
\end{remark}

In \cite{Chen21}, with strongly monotone property, Chen et al. also
considered  the global error bound of the AVEs (\ref{p3}) and
presented the following result.

\begin{theorem} (Theorem 4.1 in \cite{Chen21}) Let $x^*$ be the unique solution of AVEs (\ref{p3}). Then for any $x\in
\mathbb{R}^{n}$,
\begin{equation}\label{eq:25}
\frac{1}{\underline{C}} \|r(x)\|\leq\|x-x^{\ast}\|\leq
\overline{C}\|r(x)\|,
\end{equation}
where
\begin{equation*}
\underline{C}=\|A +I\|+\|A-I\|\ \mbox{and}\ \overline{C}=\frac{\|A
+I\|+\|A-I\|}{\sigma_{\min}(A)^{2}-1} \ \mbox{with}\
\sigma_{\min}(A)>1.
\end{equation*}
\end{theorem}

Now, we show that the error bounds in Corollary 2.1 are sharper than
that in  Theorem 2.3. Here, we consider $\|\cdot\|_{2}$ for
Corollary 2.1 and Theorem 2.3.
\begin{theorem} Let the assumptions of Corollary 2.1 and Theorem 2.3 be
satisfied. Then for any diagonal matrix $D = diag(d_{i})$ with
$d_{i}\in [-1, 1]$,
\begin{equation}\label{eq:26}
\|A +I\|_{2}+\|A-I\|_{2}\geq\max\|A-D\|_{2}
\end{equation}
and
\begin{equation}\label{eq:27}
\max\|(A-D)^{-1}\|_{2}\leq \frac{\|A
+I\|_{2}+\|A-I\|_{2}}{\sigma_{\min}(A)^{2}-1}.
\end{equation}
\end{theorem}
\begin{proof} Firstly, we prove (\ref{eq:26}). Since
\begin{align} \label{eq:28}
\max\|A-D\|_{2}&=\max\|A +I-I-D\|_{2}\nonumber\\
&\leq\max(\|A +I\|_{2}+\|I+D\|_{2})\nonumber\\
&=\|A +I\|_{2}+\max\|I+D\|_{2}\nonumber\\
&=2+\|A +I\|_{2}
\end{align}
and
\begin{align} \label{eq:29}
\max\|A-D\|_{2}\leq 2+\|A -I\|_{2}.
\end{align}
Combining (\ref{eq:28}) and (\ref{eq:29}) leads to
\[
\max\|A-D\|_{2}\leq2+\frac{\|A -I\|_{2}+\|A +I\|_{2}}{2}\leq\|A
-I\|_{2}+\|A +I\|_{2},
\]
the latter inequality holds because
\[
2=\|2I\|_{2}=\|I+A+I-A\|_{2} \leq\|I+A\|_{2}+\|I-A\|_{2}.
\]
This shows that (\ref{eq:26}) holds.

Next, we prove (\ref{eq:27}). It only needs to prove
\[
(\sigma_{\min}(A)^{2}-1)\max\|(A-D)^{-1}\|_{2} \leq\|A
+I\|_{2}+\|A-I\|_{2}.
\]
From $\sigma_{\min}(A)>1$, we know $\|A^{-1}\|_{2}<1$. Further,
$\|A^{-1}\|_{2}\|D\|_{2}<1$.  By Banach Perturbation Lemma in
\cite{Ortega}, clearly,
\[
\|(A-D)^{-1}\|_{2}\leq\frac{\|A^{-1}\|_{2}}{1-\|A^{-1}\|_{2}\|D\|_{2}}\leq\frac{\|A^{-1}\|_{2}}{1-\|A^{-1}\|_{2}}.
\]
By the simple computation, we have
\begin{align*}
\|A^{-1}\|_{2}(\sigma_{\min}(A)^{2}-1)=&\|A^{-1}\|_{2}(\frac{1}{\sigma_{\max}(A^{-1})^{2}}-1)\\
=&\|A^{-1}\|_{2}(\frac{1}{\|A^{-1}\|_{2}^{2}}-1)\\
=&\frac{1-\|A^{-1}\|_{2}^{2}}{\|A^{-1}\|_{2}}.
\end{align*}
So, we get
\[
(\sigma_{\min}(A)^{2}-1)\max\|(A-D)^{-1}\|_{2} \leq
1+\|A^{-1}\|_{2},
\]
this shows that then we only show that
\[
1+\|A^{-1}\|_{2}\leq|A +I\|_{2}+\|A-I\|_{2}.
\]
In fact, together with $\|A^{-1}\|_{2}\leq1$, we have
\begin{align*}
1+\|A^{-1}\|_{2}\leq& 2=\|2I\|_{2}=\|I+A^{-1}+I-A^{-1}\|_{2}\\
\leq&\|I+A^{-1}\|_{2}+\|I-A^{-1}\|_{2}\\
=&\|A^{-1}(I+A)\|_{2}+\|A^{-1}(I-A)\|_{2}\\
\leq&\|A^{-1}\|_{2}(\|A +I\|_{2}+\|A-I\|_{2})\\
\leq&\|A +I\|_{2}+\|A-I\|_{2}.
\end{align*}
This proves (\ref{eq:27}). \end{proof}

From the proof of Theorem 2.4, we find some interesting results for
the matrix norm, see Proposition 2.1.

\begin{prop} The following statements hold:
\begin{description}
\item  $ (1)$ For $A\in \mathbb{R}^{n\times n}$ and  $\alpha>0$ in $\mathbb{R}$,
\[
2\alpha-\|A\| \leq\|\alpha I+A\|+\|\alpha I-A\|.
\]
\item  $ (2)$ Let  $A\in \mathbb{R}^{n\times n}$ and  $\alpha\geq0$ in $\mathbb{R}$. Then for  $\|A\| \leq \alpha$,
\[
\alpha+\|A\| \leq\|\alpha I+A\|+\|\alpha I-A\|.
\]
\end{description}
\end{prop}

\subsection{Estimations of $\overline{\alpha}$, $\overline{\beta}$  and $\overline{\gamma}$}
By observing $\overline{\alpha}$, $\overline{\beta}$  and
$\overline{\gamma}$, since they contain any $D=diag(d_{i})$ with
$d_{i}\in [-1,1]$, it is difficult to directly compute quantities
$\overline{\alpha}$, $\overline{\beta}$  and $\overline{\gamma}$. To
overcome this disadvantage, in this subsection, we explore some
computable estimations for $\overline{\alpha}$, $\overline{\beta}$
and $\overline{\gamma}$.

In the following, we focus on estimating the value of
$\overline{\alpha}$. For $\overline{\beta}$, its process is
completely analogical, with regard to $\overline{\gamma}$ just for
their special case. To present the reasonable estimations for
$\overline{\alpha}$,  here we consider three aspects: (1)
$\rho(|A^{-1}B|)<1$; (2) $\sigma_{\min}(A)>\sigma_{\max}(B)$; (3)
$\sigma_{max}(A^{-1}B)<1$. For these three cases, the AVEs
(\ref{p1}) for any $b\in \mathbb{R}^{n}$ has a unique solution, see
Theorem 2 in \cite{Rohn14}, Theorem 2.1  in \cite{Wu20} and
Corollary 3.2 in \cite{Wu21b}. Similarly, for $\overline{\beta}$, we
display three aspects: (1) $\rho(|BA^{-1}|)<1$; (2)
$\sigma_{\min}(A)>\sigma_{\max}(B)$; (3) $\sigma_{max}(BA^{-1})<1$.
For these three cases, the AVEs (\ref{eq:2}) for any $b\in
\mathbb{R}^{n}$ has a unique solution as well, see Corollary 3.3  in
\cite{Wu21a},  Lemma 2.2 in \cite{Li22} and  Corollary 3.2 in
\cite{Wu21a}.

\subsubsection{Case I}

Assume that matrices $A$ and $B$ in (\ref{p1}) satisfy
\[
\rho(|A^{-1}B|)<1.
\]
We can present a reasonable estimation for $\overline{\alpha}$, see
Theorem 2.5.

\begin{theorem}
Let $\rho(|A^{-1}B|)<1$ in \emph{(\ref{p1})}. Then
\begin{equation}\label{eq:211}
\overline{\alpha}\leq\|(I-|A^{-1}B|)^{-1}\|\|A^{-1}\|.
\end{equation}
\end{theorem}
\begin{proof} Since $A^{-1}BD\leq |A^{-1}BD| \leq|A^{-1}B|$ for any $D = diag(d_{i})$ with $d_{i}\in
[-1, 1]$,  by Theorem 8.1.18 of \cite{Horn}, we get
\[
\rho(A^{-1}BD)\leq\rho(|A^{-1}BD|)\leq\rho(|A^{-1}B|)<1.
\]
So
\begin{align*}
|(I-A^{-1}BD)^{-1}|=&|I+(A^{-1}BD)+(A^{-1}BD)^{2}+...|\\
\leq&I+(|A^{-1}BD|)+(|A^{-1}BD|)^{2}+...\\
\leq&I+(|A^{-1}B|)+(|A^{-1}B|)^{2}+...\\
=&(I-|A^{-1}B|)^{-1}.
\end{align*}
Combining
\[
\|(A-BD)^{-1}\|=\|(I-A^{-1}BD)^{-1}A^{-1}\|\leq\|(I-A^{-1}BD)^{-1}\|\|A^{-1}\|
\]
and
\[
\|(I-A^{-1}BD)^{-1}\|\leq\||(I-A^{-1}BD)^{-1}|\|\leq\|(I-|A^{-1}B|)^{-1}\|,
\]
the desired bound (\ref{eq:211}) can be gained.  \end{proof}

Similar to the proof of Theorem 2.5, for $\overline{\beta}$,  we
have
\begin{theorem}
Let $\rho(|BA^{-1}|)<1$ in \emph{(\ref{eq:2})}. Then
\[
\overline{\beta}\leq\|A^{-1}\|\|(I-|BA^{-1}|)^{-1}\|.
\]
\end{theorem}

Needless to say, for $\overline{\gamma}$, we have

\begin{cor}
Let $\rho(|A^{-1}|)<1$ in \emph{(\ref{p3})}. Then
\[
\overline{\gamma}\leq\|A^{-1}\|\|(I-|A^{-1}|)^{-1}\|.
\]
\end{cor}

\subsubsection{Case II}

If $\sigma_{\min}(A)>\sigma_{\max}(B)$ in (\ref{p1}), then for
$\overline{\alpha}$ we have Theorem 2.7.
\begin{theorem}
Let $\sigma_{\min}(A)>\sigma_{\max}(B)$  in \emph{(\ref{p1})}. Then
\begin{equation}\label{eq:12}
\overline{\alpha}\leq\frac{1}{\sigma_{\min}(A)-\sigma_{\max}(B)}.
\end{equation}
\end{theorem}
\begin{proof}  For any diagonal matrix $D = diag(d_{i})$ with $d_{i}\in
[-1, 1]$,
\[
\sigma_{\min}(A)-\sigma_{\max}(BD)\geq\sigma_{\min}(A)-\sigma_{\max}(B)\sigma_{\max}(D)\geq\sigma_{\min}(A)-\sigma_{\max}(B)
\]
and
\[
\sigma_{\min}(A-BD)\geq\sigma_{\min}(A)-\sigma_{\max}(BD).
\]
So
\begin{align*}
\|(A-BD)^{-1}\|_{2}=&\sigma_{\max}((A-BD)^{-1})\\
=&\frac{1}{\sigma_{\min}(A-BD)}\\
\leq&\frac{1}{\sigma_{\min}(A)-\sigma_{\max}(BD)}\\
\leq&\frac{1}{\sigma_{\min}(A)-\sigma_{\max}(B)}.
\end{align*}
This completes the proof for Theorem 2.7. \end{proof}

For $\overline{\beta}$,  we have the same as the result presented in
Theorem 2.8.

\begin{theorem}
Let $\sigma_{\min}(A)>\sigma_{\max}(B)$ in \emph{(\ref{eq:2})}. Then
\[
\overline{\beta}\leq\frac{1}{\sigma_{\min}(A)-\sigma_{\max}(B)}.
\]
\end{theorem}

\begin{cor}
Let $\sigma_{\min}(A)>1$ in \emph{(\ref{p3})}. Then
\[
\overline{\gamma}\leq\frac{1}{\sigma_{\min}(A)-1}.
\]
\end{cor}

It is easy to see that the upper bound in Corollary 2.3 is still
sharper than that in Theorem 2.3, i.e.,
\[
\frac{1}{\sigma_{\min}(A)-1}\leq \frac{\|A
+I\|_{2}+\|A-I\|_{2}}{\sigma_{\min}(A)^{2}-1},
\]
which is equal to
\[
\sigma_{\min}(A)+1\leq \|A +I\|_{2}+\|A-I\|_{2}.
\]
In fact, by using Proposition 2.1, we have
\[
1+\|A^{-1}\|_{2}\leq\|I+A^{-1}\|_{2}+\|I-A^{-1}\|_{2}\leq\|A^{-1}\|_{2}(\|A
+I\|_{2}+\|A-I\|_{2}).
\]

\begin{remark}
The conditions in Theorems 2.5 and 2.7 are not included each other,
e.g., taking
\[
A=\left[\begin{array}{ccc}
1&0\\
0&1\\
\end{array}\right], B=\left[\begin{array}{ccc}
0.9&-0.4\\
0.4&0.9\\
\end{array}\right].
\]
By the simple computation,
\[
\sigma_{\min}(A)=1>\sigma_{\max}(B)=0.9849,
\]
but
\[
\rho(|A^{-1}B|)=1.3>1.
\]
This shows that matrices $A$ and $B$ satisfy the condition of
Theorem 2.7, do not satisfy the condition of  Theorem 2.5. Now we
take
\[
A=\left[\begin{array}{ccc}
2&1\\
0&2\\
\end{array}\right],\ B=\left[\begin{array}{ccc}
1.6&0\\
0&1.6\\
\end{array}\right].
\]
By the simple computation,
\[
\rho(|A^{-1}B|)=0.8000<1,
\]
but
\[
\sigma_{\min}(A)= 1.5616<\sigma_{\max}(B)=1.6.
\]
This shows that matrices $A$ and $B$ satisfy the condition of
Theorem 2.5, do not satisfy the condition of Theorem 2.7.
\end{remark}

\subsubsection{Case III}
Here, we consider this case that $B$ is nonsingular and
$\sigma_{max}(A^{-1}B)<1$ in (\ref{p1}). Based on this, for
$\overline{\alpha}$ we have Theorem 2.9.

\begin{theorem}
Let  $B$ be nonsingular and $\sigma_{max}(A^{-1}B)<1$  in
\emph{(\ref{p1})}. Then
\begin{equation}\label{eq:213}
\overline{\alpha}\leq\frac{\|A^{-1}B\|_{2}\|B^{-1}\|_{2}}{1-\|A^{-1}B\|_{2}}.
\end{equation}
\end{theorem}
\begin{proof} From Corollary 3.2 in \cite{Wu21b}, $A+BD$ is
nonsingular for any diagonal matrix $D = diag(d_{i})$ with $d_{i}\in
[-1, 1]$ under the assumptions. So, we have
\[
\|(A-BD)^{-1}\|_{2}=\|(B^{-1}A-D)^{-1}B^{-1}\|_{2}\leq\|(B^{-1}A-D)^{-1}\|_{2}\|B^{-1}\|_{2}.
\]
Noting that  $\sigma_{max}(A^{-1}B)<1$ is equal to
$\|A^{-1}B\|_{2}<1$ and
\[
\|A^{-1}BD\|_{2}\leq\|A^{-1}B\|_{2}\|D\|_{2}<1.
\]
Making use of Banach perturbation lemma in \cite{Ortega} leads to
\begin{align*}
\|(B^{-1}A-D)^{-1}\|_{2}\leq&\frac{\|A^{-1}B\|_{2}}{1-\|A^{-1}B\|_{2}\|D\|_{2}}\leq\frac{\|A^{-1}B\|_{2}}{1-\|A^{-1}B\|_{2}}.
\end{align*}
Therefore, the proof of Theorem 2.9 is completed.  \end{proof}

For $\overline{\beta}$,  we have

\begin{theorem}
Let $B$ be nonsingular and $\sigma_{max}(BA^{-1})<1$ in
\emph{(\ref{eq:2})}. Then
\[
\overline{\beta}\leq\frac{\|B^{-1}\|_{2}\|BA^{-1}\|_{2}}{1-\|BA^{-1}\|_{2}}.
\]
\end{theorem}

\begin{cor}
Let $\sigma_{max}(A^{-1})<1$   in \emph{(\ref{p3})}. Then
\[
\overline{\gamma}\leq\frac{\|A^{-1}\|_{2}}{1-\|A^{-1}\|_{2}}.
\]
\end{cor}

Usually, Corollary 2.4 is equal to Corollary 2.3.

\begin{remark}
Comparing Theorem 2.7 with Theorem 2.9, it is easy to find that
\[
\sigma_{max}(A^{-1}B)\leq\sigma_{max}(A^{-1})\sigma_{max}(B)=\frac{\sigma_{max}(B)}{\sigma_{min}(A)}.
\]
Whereas, it does not show that Theorem 2.9 is weaker than Theorem
2.7 because Theorem 2.9 asks for $B$ being nonsingular. Besides, we
need to point out that  Theorem 2.9 sometimes performs better than
Theorem 2.7, vice versa. To illustrate this, we take
\[
A=\left[\begin{array}{ccc}
2&0\\
0&3\\
\end{array}\right],\ B=\left[\begin{array}{ccc}
1&0\\
0&1.5\\
\end{array}\right].
\]
Obviously, $B$ is nonsingular,
\[
\sigma_{max}(A^{-1}B)=0.5<1, \sigma_{min}(A)=2>\sigma_{max}(B)=1.5.
\]
This shows that the conditions of Theorem 2.7 and Theorem 2.9 are
satisfied. From Theorem 2.7 and Theorem 2.9, we have
\[
\frac{\|A^{-1}B\|_{2}\|B^{-1}\|_{2}}{1-\|A^{-1}B\|_{2}}=
1<\frac{1}{\sigma_{min}(A)-\sigma_{\max}(B)}=2,
\]
from which  shows that the upper bound in Theorem 2.9 is sharper
than that in Theorem 2.7. Now, we take
\[
A=\left[\begin{array}{ccc}
2&0\\
0&3\\
\end{array}\right],\ B=\left[\begin{array}{ccc}
1.5&0\\
0&1\\
\end{array}\right].
\]
Likewise, $B$ is nonsingular,
\[
\sigma_{max}(A^{-1}B)=0.75<1, \sigma_{min}(A)=2>\sigma_{max}(B)=1.5.
\]
This implies that the conditions of Theorem 2.7 and Theorem 2.9 are
satisfied as well. From Theorem 2.7 and Theorem 2.9, we have
\[
\frac{\|A^{-1}B\|_{2}\|B^{-1}\|_{2}}{1-\|A^{-1}B\|_{2}}=
5>\frac{1}{\sigma_{min}(A)-\sigma_{\max}(B)}=2,
\]
which implies that the upper bound in Theorem 2.7 is sharper than
that in Theorem 2.9.
\end{remark}

\begin{remark} Although  Zamani and Hlad\'{\i}k
in \cite{Zamani} presented some estimations for $\overline{\gamma}$,
their results are either impossible to achieve or difficult to
calculate, see Theorem 8 and Proposition 11 in \cite{Zamani},
respectively.

\end{remark}
\section{Perturbation bound}
In this section, we focus on the perturbation analysis of AVEs when
 $A, B$ and $b$ are perturbed. For instance, for the AVEs
(\ref{p1}), when $\Delta A$, $\Delta B$ and $\Delta b$ are the
perturbation terms of $ A$, $B$ and $b$, respectively, how do we
characterize the change in the solution of the following perturbed
AVEs (\ref{p4})
\begin{equation}\label{p4}
(A+\Delta A)y-(B+\Delta B)|y|=b+\Delta b.
\end{equation}

For the AVEs (\ref{p1}), firstly,  we consider the following special
case
\begin{equation}\label{p5}
Ay-B|y|=b+\Delta b.
\end{equation}
Assume that the AVEs (\ref{p5}) has the unique solution $y^{\ast}$.
Let $x^{\ast}$ be the unique solution of AVEs (\ref{p1}).
Subtracting (\ref{p1}) from (\ref{p5}), we have
\[
Ax^{\ast}-B|x^{\ast}|- (Ay^{\ast}-B|y^{\ast}|)=-\Delta b,
\]
which is equal to
\begin{equation}\label{eq:33}
x^{\ast}-y^{\ast}=-(A-B\tilde{D})^{-1}\Delta b,
\end{equation}
where $\tilde{D}=diag(\tilde{d}_{i})$ with $\tilde{d}_{i}\in
[-1,1]$. Making use of the norm for both sides of (\ref{eq:33}),
noting that
\begin{equation}\label{eq:34}
\|(A-B\tilde{D})^{-1}\|\leq\max \|(A-BD)^{-1}\|
\end{equation}
for any $D=diag(d_{i})$ with $d_{i}\in [-1,1]$,  we have
\begin{equation}\label{eq:35}
\|x^{\ast}-y^{\ast}\|\leq\max \|(A-BD)^{-1}\|\|\Delta b\|.
\end{equation}
Moreover, from the AVEs (\ref{p1}) with $x^{\ast}$, it is easy to
check that
\begin{equation}\label{eq:36}
\frac{\|b\|}{\|x^{\ast}\|}\leq\|A\|+\|B\|.
\end{equation}
Combining (\ref{eq:35}) with (\ref{eq:36}), for any $D=diag(d_{i})$
with $d_{i}\in [-1,1]$, we obtain
\begin{align*}
\frac{\|x^{\ast}-y^{\ast}\|}{\|x^{\ast}\|}\leq&\frac{\max
\|(A-BD)^{-1}\|\|\Delta
b\|}{\|x^{\ast}\|}=\frac{\max\|(A-BD)^{-1}\|\|\Delta
b\|}{\|b\|}\frac{\|b\|}{\|x^{\ast}\|}\\
\leq& \frac{\max \|(A-BD)^{-1}\|\|\Delta b\|}{\|b\|}(\|A\|+\|B\|),
\end{align*}
from which we immediately get

\begin{theorem}
Let $x^{\ast}$, $y^{\ast}$ be the unique solutions of AVEs
(\ref{p1}) and (\ref{p5}), respectively. Then
\begin{align*}
\frac{\|x^{\ast}-y^{\ast}\|}{\|x^{\ast}\|} \leq
\frac{\overline{\alpha}\|\Delta b\|}{\|b\|}(\|A\|+\|B\|),
\end{align*}
where $\overline{\alpha}$ is defined in Theorem 2.1.
\end{theorem}

Similarly, we assume that  $y^{\ast}$ is the unique solution of AVEs
(\ref{p4}). The following theorem is the framework of AVEs
(\ref{p1}) perturbation.

\begin{theorem}
Let $x^{\ast}$, $y^{\ast}$ be the unique solutions of AVEs
(\ref{p1}) and (\ref{p4}), respectively. Then
\begin{equation}\label{eq:37}
\frac{\|x^{\ast}-y^{\ast}\|}{\|x^{\ast}\|}
\leq\mu_{1}\bigg(\frac{\|\Delta b\|}{\|b\|}(\|A\|+\|B\|)+\|\Delta
A\|+\|\Delta B\|\bigg).
\end{equation}
where
\begin{equation*}
\mu_{1}=\max \{\|(A-BD+(\Delta A-\Delta BD ))^{-1}\|:D=diag(d_{i})\
\mbox{with}\ d_{i}\in [-1,1]\}.
\end{equation*}
\end{theorem}
\begin{proof} Based on the assumptions, the AVEs (\ref{p1}) is
equal to
\begin{equation}\label{eq:38}
(A+\Delta A)x^{\ast}-(B+\Delta B)|x^{\ast}|=b+\Delta
Ax^{\ast}-\Delta B|x^{\ast}|.
\end{equation}
Subtracting (\ref{eq:38}) from (\ref{p4}) with $y^{\ast}$, we have
\begin{equation}\label{eq:n39}
(A-B\tilde{D}+(\Delta A-\Delta B\tilde{D}
))(x^{\ast}-y^{\ast})=-\Delta b+\Delta Ax^{\ast}-\Delta B|x^{\ast}|,
\end{equation}
where $\tilde{D}=diag(\tilde{d}_{i})$ with $\tilde{d}_{i}\in
[-1,1]$.  Using the norm for (\ref{eq:n39}), similar to
(\ref{eq:35}), results in
\begin{equation}\label{eq:310}
\|x^{\ast}-y^{\ast}\|\leq\mu_{1}(\|\Delta b\|+(\|\Delta A\|+\|\Delta
B\|)\|x^{\ast}\|).
\end{equation}
By making use of
(\ref{eq:310}) and (\ref{eq:36}), the desired bound (\ref{eq:37})
can be obtained. \end{proof}

\begin{remark}
Here, a special case is considered, i.e, when $B=0$ in (\ref{p1}),
the AVEs (\ref{p1}) reduces to the linear systems
\[
Ax=b.
\]
It is not difficult to find that from Theorem 3.2 we easily obtain a
classical perturbation bound for the linear systems $Ax=b$, i.e.,
\[\frac{\|x^{\ast}-y^{\ast}\|}{\|x^{\ast}\|}
\leq \frac{\kappa(A)}{1-\frac{\kappa(A)\|\Delta
A\|}{\|A\|}}\Bigg(\frac{\|\Delta b\|}{\|b\|}+\frac{\|\Delta
A\|}{\|A\|}\Bigg)
\]
for $\|A^{-1}\|\|\Delta A\| < 1$, where $\kappa(A)=\|A^{-1}\|\|A\|$,
see Theorem 1.3 in numerical linear algebra textbooks \cite{Sun01}
and Theorem 2.1 \cite{Skeel79}. In fact, in such case, the right
side in (\ref{eq:37}) reduces to
\[
\|(A+\Delta A)^{-1}\|\bigg(\frac{\|\Delta b\|}{\|b\|}(\|A\|+\|\Delta
A\|\bigg).
\]
Noting that $\|A^{-1}\|\|\Delta A\| <1$, by the Banach perturbation
lemma in \cite{Ortega}, we have
\begin{align*}
\|(A+\Delta A)^{-1}\|\bigg(\frac{\|\Delta b\|}{\|b\|}(\|A\|+\|\Delta
A\|\bigg)&=\|(A+\Delta A)^{-1}\|\|A\|\bigg(\frac{\|\Delta
b\|}{\|b\|}+\frac{\|\Delta A\|}{\|A\|}\bigg)\\
&\leq\frac{\|A^{-1}\|}{1-\|A^{-1}\|\|\Delta
A\|}\|A\|\bigg(\frac{\|\Delta b\|}{\|b\|}+\frac{\|\Delta
A\|}{\|A\|}\bigg)\\
&=\frac{\kappa(A)}{1-\frac{\kappa(A)\|\Delta
A\|}{\|A\|}}\Bigg(\frac{\|\Delta b\|}{\|b\|}+\frac{\|\Delta
A\|}{\|A\|}\Bigg).
\end{align*}

\end{remark}

\begin{remark} Similar to the AVEs (\ref{p1}), some perturbation
bounds for the AVEs (\ref{eq:2}) and (\ref{p3}) can be obtained.
Here is omitted.
\end{remark}

\begin{remark}
By directly making use of the results in  the Section 2.2,
naturally, we can present some computable upper bounds for $\mu_{1}$
in Theorem 3.2.
\end{remark}

\section{An application}

As an application, in this section, a new approach for the existing
perturbation bounds of LCP in \cite{Chen07}. As is known, for the
LCP, without making use of change of variable, it is also
transformed into a certain absolute value equations by the fact that
\[
w=Mz+q, z\geq 0,w \geq 0, z^{T}w = 0 \Leftrightarrow \min\{z,
Mz+q\}=0,
\]
see \cite{Chen06, Zhang09}. Noting that
\[\min\{a,b\}=\frac{1}{2}(a+b-|a-b|)\ \forall a,b \in
\mathbb{R}^{n},
\]
obviously, the LCP $(M,q)$ is equal to the following absolute value
equations
\begin{equation}\label{eq:422}
\tfrac{1}{2}((M+I)z+q)=\tfrac{1}{2}|(M-I)z+q|
\end{equation}
by setting $a=Mz+q$ and $b=z$.

In fact, by making use of the AVEs (\ref{eq:422}), we also obtain
Theorem 2.8 and Theorem 3.1 in \cite{Chen07}. To achieve this goal,
without loss of generality, we consider the relationship between the
solution of the LCP$(A,b)$ and the solution of LCP$(B,c)$, i.e.,
\[
\min\{x, Ax+b\}=0,\ \mbox{and}\  \min\{x, Bx+c\}=0.
\]
Assume that $x^{\ast}$ and $y^{\ast}$ are the unique solutions of
the LCP$(A,b)$ and the LCP$(B,c)$, respectively. Similar to
(\ref{eq:422}), we have
\begin{equation}\label{eq:423}
\tfrac{1}{2}((A+I)x^{\ast}+b)=\tfrac{1}{2}|(A-I)x^{\ast}+b|
\end{equation}
and
\begin{equation}\label{eq:424}
\tfrac{1}{2}((B+I)y^{\ast}+c)=\tfrac{1}{2}|(B-I)y^{\ast}+c|.
\end{equation}
Subtracting  (\ref{eq:423}) from (\ref{eq:424}), together with Lemma
2.1, we have
\[
\frac{I+\tilde{D}}{2}(x^{\ast}-y^{\ast})+\frac{I-\tilde{D}}{2}Ax^{\ast}=\frac{(I-\tilde{D})B}{2}y^{\ast}-\frac{(I-\tilde{D})(b-c)}{2},
\]
which is equal to
\[
\frac{I+\tilde{D}}{2}(x^{\ast}-y^{\ast})+\frac{(I-\tilde{D})Ax^{\ast}}{2}-\frac{(I-\tilde{D})Ay^{\ast}}{2}
=-\frac{(I-\tilde{D})Ay^{\ast}}{2}+\frac{(I-\tilde{D})By^{\ast}}{2}-\frac{(I-\tilde{D})(b-c)}{2},
\]
or,
\begin{equation}\label{eq:44}
\frac{I+\tilde{D}+(I-\tilde{D})A}{2}(x^{\ast}-y^{\ast})=-\frac{(I-\tilde{D})(A-B)}{2}y^{\ast}-\frac{(I-\tilde{D})(b-c)}{2},
\end{equation}
where $\tilde{D}=diag(\tilde{d}_{i})$ with $\tilde{d}_{i}\in
[-1,1]$. We take
$\tilde{\Lambda}=\frac{I-\tilde{D}}{2}=diag(\tilde{\lambda}_{i})$
with $\tilde{\lambda}_{i}\in [0,1]$, and from (\ref{eq:44}) get
\begin{equation}\label{eq:425}
(I-\tilde{\Lambda}+\tilde{\Lambda}
A)(x^{\ast}-y^{\ast})=-\tilde{\Lambda}(A-B)y^{\ast}-\tilde{\Lambda}(b-c).
\end{equation}
From (\ref{eq:425}), we have
\begin{align} \label{eq:426}
\|x^{\ast}-y^{\ast}\|=&\|(I-\tilde{\Lambda}+\tilde{\Lambda}
A)^{-1}\tilde{\Lambda}((A-B)y^{\ast}+b-c)\|\nonumber\\
\leq&\|(I-\tilde{\Lambda}+\tilde{\Lambda}
A)^{-1}\tilde{\Lambda}\|\|(A-B)y^{\ast}+b-c\|\nonumber\\
\leq&\|(I-\tilde{\Lambda}+\tilde{\Lambda}
A)^{-1}\tilde{\Lambda}\|(\|A-B\|\|y^{\ast}\|+\|b-c\|)\nonumber\\
\leq&\beta(A)(\|A-B\|\|y^{\ast}\|+\|b-c\|),
\end{align}
where $\beta(A)=\max\{\|(I-\Lambda+\Lambda
A)^{-1}\Lambda\|:\Lambda=diag(\lambda_{i})~ \mbox{with}~ \lambda_{i}
\in[0,1]\}$, or
\begin{align} \label{eq:427}
\|x^{\ast}-y^{\ast}\| \leq&\beta(B)(\|A-B\|\|x^{\ast}\|+\|b-c\|),
\end{align}
where $\beta(B)=\max\{\|(I-\Lambda+\Lambda
B)^{-1}\Lambda\|:\Lambda=diag(\lambda_{i})~ \mbox{with}~ \lambda_{i}
\in[0,1]\}$.

Noting that 0 is the solution of LCP$(B,c_{+})$, repeating the above
process, we obtain
\begin{equation}\label{eq:428}
\|y^{\ast}\|\leq \beta(B)\|(-c)_{+}\|.
\end{equation}
Furthermore, from $Ax^{\ast}+b\geq0$, we obtain
$(-b)_{+}\leq(Ax^{\ast})_{+}\leq|Ax^{\ast}|$, which implies that
\begin{equation}\label{eq:429}
\frac{1}{\|x^{\ast}\|}\leq\frac{\|A\|}{\|(-b)_{+}\|}.
\end{equation}
Hence,  combining  (\ref{eq:426}), (\ref{eq:427}), (\ref{eq:428})
with (\ref{eq:429}), we have

\begin{lemma}
Assume that $x^{\ast}$ and $y^{\ast}$ are the unique solutions of
the LCP$(A,b)$ and the LCP$(B,c)$, respectively. Let
\begin{equation*}\beta(A)=\max\{\|(I-\Lambda+\Lambda A)^{-1}\Lambda\|:\Lambda=diag(\lambda_{i})~ \mbox{with}~ \lambda_{i} \in[0,1]\}
\end{equation*} and
\begin{equation*}\beta(B)=\max\{\|(I-\Lambda+\Lambda B)^{-1}\Lambda\|:\Lambda=diag(\lambda_{i})~ \mbox{with}~ \lambda_{i}
\in[0,1]\}.
\end{equation*} 
Then
\begin{equation*}\label{eq:430}
\|x^{\ast}-y^{\ast}\| \leq
\beta(A)(\beta(B)\|A-B\|\|(-c)_{+}\|+\|b-c\|)
\end{equation*}
and
\begin{equation*}\label{eq:431}
\frac{\|x^{\ast}-y^{\ast}\| }{\|x^{\ast}\|}\leq
\beta(B)\Big(\|A-B\|+\frac{\|b-c\|\|A\|}{\|(-b)_{+}\|}\Big).
\end{equation*}
\end{lemma}

Let
\begin{equation}\label{eq:432}
\mathcal{ M}:=\{A|\ \beta(M)\|M-A\|\leq\eta<1\},
\end{equation}
where
\[
\beta(M)=\max\|(I-\Lambda+\Lambda M)^{-1}\Lambda\|.
\]
Then for any $A\in \mathcal{ M}$, we have
\begin{equation}\label{eq:433}
\beta(A)\leq \alpha(M)=\frac{1}{1-\eta}\beta(M),
\end{equation}
see \cite{Chen07}. Together with Lemma 4.1, Theorem 4.1 can be
obtained, i.e., Theorems 2.8 and 3.1 in \cite{Chen07}.

\begin{theorem} (Theorems 2.8 and 3.1 in \cite{Chen07})
Assume that $x^{\ast}$ and $y^{\ast}$ are the unique solutions of
the LCP$(A,b)$ and the LCP$(B,c)$, respectively. Then for any
$A,B\in \mathcal{ M}$, where $\mathcal{ M}$ is defined as in
(\ref{eq:432}),
\begin{equation*}\label{eq:434}
\|x^{\ast}-y^{\ast}\| \leq
\alpha(M)^{2}\|A-B\|\|(-c)_{+}\|+\alpha(M)\|b-c\|,
\end{equation*}
where $\alpha(M)$ is defined as in (\ref{eq:433}). Further, let
$A=M$, $B=M+\Delta M$, $b=q$, $c=q+\Delta q$ with $\|\Delta M\|\leq
\epsilon\| M\|$ and $ \|\Delta q\|\leq \epsilon\|(-q)_{+}\|$,
$\epsilon \beta(M)\|M\|=\delta<1$. Then
\begin{equation*}\label{eq:435}
\frac{\|x^{\ast}-y^{\ast}\|
}{\|x^{\ast}\|}\leq\frac{2\delta}{1-\delta}.
\end{equation*}
\end{theorem}

\section{Numerical examples}
In this section, we focus on investigating the feasibility of the
relative perturbation bounds of Theorem 3.2 in the Section 3 by
making use of some numerical examples. For the sake of convenience,
we consider the AVEs (\ref{p1})
\[
Ax-B|x|=b,\ \mbox{with} \ A=I+M, B=I-M  \ \mbox{and}\ b=-q,
\]
which is from the LCP (\ref{eq:14}) by making use of change of
variable $z=|x|+x$ and $w=|x|-x$. For the convenient comparison,
here we introduce some  notations, i.e.,
\[
r=\frac{\|x^{\ast}-y^{\ast}\|_{2}}{\|x^{\ast}\|_{2}},
\tau=\|(I-|\mathcal{A}^{-1}\mathcal{B}|)^{-1}\|_{2}\|\mathcal{A}^{-1}\|_{2}w,
\]
\[
\upsilon=\frac{w}{\sigma_{\min}(\mathcal{A})-\sigma_{\max}(\mathcal{B})},
\nu=\frac{\|\mathcal{A}^{-1}\mathcal{B}\|_{2}\|\mathcal{B}^{-1}\|_{2}w}{1-\|\mathcal{A}^{-1}\mathcal{B}\|_{2}},
\]
where
\[
\mathcal{A}=A+\Delta A,\mathcal{B}=B+\Delta B, w=\frac{\|\Delta
b\|_{2}}{\|b\|_{2}}(\|A\|_{2}+\|B\|_{2})+\|\Delta A\|_{2}+\|\Delta
B\|_{2},
\]
$\tau$, $\upsilon$ and $\nu$ in order denote the upper bounds of the
relative  perturbation bounds of Theorem 3.2 under Theorems 2.5, 2.7
and 2.9, $x^{\ast}$ and $y^{\ast}$ in order denote the solutions of
the corresponding AVEs and the perturbed AVEs and can be obtained by
using the iteration method on page 364 in \cite{Mangasarian06} with
the initial vector being zero and the corresponding absolute error
less than $10^{-6}$. All the computations are done in Matlab R2021b
on an HP PC (Intel@ Celeron@ G4900, 3.10GHz, 8.00 GB of RAM).

\textbf{Example 5.1} \cite{Ahn83}  Let
\[
M=\mbox{tridiag}(1,4,-2)\in \mathbb{R}^{n\times n}, q=-4e\in
\mathbb{R}^{n}.
\]
Obviously,  $M$ is an $H_{+}$-matrix. This implies that the
corresponding LCP has a unique solution, i.e., the equal AVEs has a
unique solution.

One is interested in the perturbation error for the solution of the
corresponding AVEs (\ref{p1}) caused by small changes in $A$, $B$
and $b$. The perturbation way we do it is
\[
\Delta A=\epsilon\mbox{tridiag}(1,2,-1), \Delta
B=\epsilon\mbox{tridiag}(1,1,1), \Delta b=\epsilon e.
\]

%
%
%

\begin{table}[!htb] \centering
\begin{tabular}
{p{10pt}p{55pt}p{55pt}p{55pt}p{55pt}p{55pt}} \hline
$\epsilon$&0.01&0.015&0.02 &0.025&0.03
\\\hline
$r$&  0.0040 &   0.0061  &  0.0081  &  0.0101  &  0.0122 \\
$\tau$&  0.2845 &   0.4124 &   0.5321  &  0.6442  &  0.7495\\
$\upsilon$& 0.0465  &  0.0692 &   0.0916 &   0.1138 &   0.1356 \\
$\nu$& 0.1104 &   0.1650  & 0.2194   & 0.2736 &    0.3275\\
\hline
\end{tabular}
\\ \caption{Relative perturbation bounds of Example 5.1 with $n=30$.}
\end{table}

\begin{table}[!htb] \centering
\begin{tabular}
{p{10pt}p{55pt}p{55pt}p{55pt}p{55pt}p{55pt}} \hline
$\epsilon$&0.01&0.015&0.02 &0.025&0.03
\\\hline
$r$&0.0041 &   0.0061 &   0.0082 &   0.0102 &   0.0123 \\
$\tau$&   0.2956&    0.4282&    0.5519  &  0.6677  &  0.7762\\
$\upsilon$&  0.0473 &   0.0704 &   0.0932  &  0.1157 &   0.1378\\
$\nu$&   0.1112  &  0.1663&   0.2211&    0.2757 &   0.3301 \\
\hline
\end{tabular}
\\ \caption{Relative perturbation bounds of Example 5.1 with $n=40$.}
\end{table}

In Tables 1 and 2, we list the numerical results for three relative
perturbation bounds with the different size and $\epsilon$, from
which we find that among three relative bounds, $\upsilon<\nu$ and
$\upsilon<\tau$, i.e., $\upsilon$ is closet to the true relative
error. In addition, these numerical results also illustrate that the
proposed bounds are very close to the real relative value when the
perturbation is very small. From Tables 1 and 2, it is easy to find
that the relative perturbation bound given by Theorem 3.2 is
feasible and effective under some suitable conditions.

Next, we investigate another example, which is from \cite{Bai10}.

\textbf{Example 5.2} \cite{Bai10} Let $M=\hat{M}+\mu I$,
$q=-Mz^{\ast}$, where $\mu =4$,
\[\hat{M}=\mbox{blktridiag}(-I,S,-I)\in \mathbb{R}^{n\times n},
S=\mbox{tridiag}(-1,4,-1)\in \mathbb{R}^{m\times m}, n=m^{2}, \] and
$z^{\ast}=(1,2,1,2,\ldots,1,2,\ldots)^{T}\in \mathbb{R}^{n}$ is the
unique solution of the corresponding LCP. Of course, the equal AVEs
has a unique solution as well.

For Example 5.2, the perturbation way we do it is
\[
\Delta A=\epsilon\mbox{tridiag}(-1,2,-1), \Delta
B=\epsilon\mbox{tridiag}(1,-1,1), \Delta b=\epsilon e.
\]

\begin{table}[!htb] \centering
\begin{tabular}
{p{10pt}p{55pt}p{55pt}p{55pt}p{55pt}p{55pt}} \hline
$\epsilon$&0.01&0.015&0.02 &0.025&0.03
\\\hline
$r$& 0.0028&    0.0042&    0.0056&    0.0070&    0.0084  \\
$\tau$&  0.2571 &   0.3870&    0.5177&    0.6493 &   0.7817   \\
$\upsilon$& 0.0422&    0.0631&    0.0839 &   0.1046&    0.1252  \\
$\nu$& 0.1731&    0.2598&    0.3465&    0.4334&    0.5203 \\
\hline
\end{tabular}
\\ \caption{Relative perturbation bounds of Example 5.2 with $n= 225$.}
\end{table}

\begin{table}[!htb] \centering
\begin{tabular}
{p{10pt}p{55pt}p{55pt}p{55pt}p{55pt}p{55pt}} \hline
$\epsilon$&0.01&0.015&0.02 &0.025&0.03
\\\hline
$r$&  0.0030  &  0.0045  &  0.0060   & 0.0075  &  0.0090    \\
$\tau$&    0.2798&    0.4212&    0.5637&    0.7071  &  0.8516 \\
$\upsilon$&   0.0466&    0.0697&    0.0927&    0.1155 &   0.1382  \\
$\nu$&    0.1835&     0.2754&     0.3674&     0.4595&     0.5517   \\
\hline
\end{tabular}
\\ \caption{Relative perturbation bounds of Example 5.2 with $n=400$.}
\end{table}

Similarly, with the different size and $\epsilon$, Tables 3 and 4
list these three relative perturbation bounds for Example 5.2. From
Tables 3 and 4, we can draw the same conclusion. That is to say,
among these three bounds, $\upsilon$ is optimal, compared with other
two bounds.  These numerical results in Tables 3 and 4 further
illustrate that when the perturbation term is very small, the
proposed bounds are very close to the real relative value.
Meanwhile, this further confirms that under some suitable
conditions, Theorem 3.2 indeed provides some valid relative
perturbation bounds.

\section{Conclusion}
In this paper, by introducing a class of absolute value functions,
the frameworks of error and perturbation bounds for two types of
absolute value equations (AVEs) have been established. Besides,
without limiting the matrix type, some computable estimates for the
above upper bounds are given. As an application, a new approach for
some existing perturbation bounds of the LCP in \cite{Chen07}  is
provided. In addition, some numerical examples for absolute value
equations from the LCP are given to show the feasibility of the
proposed perturbation bounds.

%
%
%

%

\end{document}